\documentclass[10pt,oneside,a4paper,draft]{article} 	

\usepackage[fleqn]{amsmath} 						
\usepackage{amsthm,amsfonts,latexsym,amssymb,amscd} 	
\usepackage[mathscr]{eucal}   						
\usepackage[draft=false]{hyperref} 							

\newcommand{\B}{\mathcal{B}}

\newcommand{\K}{\mathcal{K}} 

\newcommand{\CC}{{\mathbb{C}}}

\newcommand{\MM}{{\mathbb{M}}}

\newcommand{\RR}{{\mathbb{R}}} 
\newcommand{\TT}{{\mathbb{T}}}

\newcommand{\ZZ}{{\mathbb{Z}}}

\newcommand{\As}{{\mathscr{A}}}\newcommand{\Bs}{{\mathscr{B}}}\newcommand{\Cs}{{\mathscr{C}}}

\newcommand{\Es}{{\mathscr{E}}}\newcommand{\Fs}{{\mathscr{F}}}

\newcommand{\Ks}{{\mathscr{K}}} 
\newcommand{\Ls}{{\mathscr{L}}}	
\newcommand{\Ms}{{\mathscr{M}}}\newcommand{\Ns}{{\mathscr{N}}}
\newcommand{\Qs}{{\mathscr{Q}}}
\newcommand{\Ts}{\mathscr{T}} 			

\DeclareFontFamily{U}{rsfs}{\skewchar\font127 }
\DeclareFontShape{U}{rsfs}{m}{n}{%
   <5> <6> rsfs5
   <7> rsfs7
   <8> <9> <10> <10.95> <12> <14.4> <17.28> <20.74> <24.88> rsfs10
}{}
\DeclareSymbolFont{rsfs}{U}{rsfs}{m}{n} 
\DeclareSymbolFontAlphabet{\scr}{rsfs}

\newcommand{\Af}{\scr{A}}

\newcommand{\Mf}{\scr{M}}

\DeclareMathOperator{\CCl}{\CC l}		
\DeclareMathOperator{\ke}{Ker}
\DeclareMathOperator{\im}{Im}
 
\DeclareMathOperator{\id}{Id} 

\DeclareMathOperator{\Ob}{Ob}

\DeclareMathOperator{\Hom}{Hom}

\DeclareMathOperator{\Span}{span}

\renewcommand{\emph}{\textbf} 										
\newcommand{\cj}[1]{\overline{#1}}									
\newcommand{\mip}[2]{(#1\mid #2)}									
\newcommand{\ip}[2]{\langle #1\mid #2\rangle}					
\newcommand{\imp}{\Rightarrow}										

\newcommand{\hlink}[2]{\href{#1}{\texttt{#2}}} 

\newcommand{\xqedhere}[2]{%
  \rlap{\hbox to#1{\hfil\llap{\ensuremath{#2}}}}}

\newcommand{\xqed}[1]{%
  \leavevmode\unskip\penalty9999 \hbox{}\nobreak\hfill
  \quad\hbox{\ensuremath{#1}}}

\theoremstyle{plain}
\newtheorem{theorem}{Theorem}[section]							

\newtheorem{proposition}[theorem]{Proposition}

\newtheorem{definition}[theorem]{Definition}

\theoremstyle{definition} 
\newtheorem{remark}[theorem]{Remark}
\newtheorem{example}[theorem]{Example}  

\numberwithin{equation}{section}  		
\title{\textbf{Kre\u\i n C*-modules}
}

\author{\normalsize Paolo Bertozzini$^a$, 
Kasemsun Rutamorn$^b$ 
\\
\normalsize $^a$\textit{Department of Mathematics and Statistics, Faculty of Science and Technology,}
\\
\normalsize \textit{Thammasat University, Bangkok 12121, Thailand} 
\\ 
\normalsize e-mail: \texttt{paolo.th@gmail.com}
\\
\normalsize $^b$\textit{Department of Mathematics, Faculty of Education,  Dhonburi Rajabhat University,}
\\
\normalsize \textit{Bangkok 10600, Thailand}
\\
\normalsize e-mail: \texttt{kasemamorn270@hotmail.com}
}

\date{\normalsize{10 December 2013}\footnote{This is a reformatted version for arXiv of a paper published in 
Chamchuri Journal of Mathematics (2013) 5:23-44. 
}
}

\begin{document}

\maketitle

\begin{abstract} \noindent 
We introduce a notion of Kre\u\i n C*-module over a C*-algebra and more generally over a Kre\u\i n \hbox{C*-al}\-gebra. 
Some properties of Kre\u\i n C*-modules and their categories are investigated. 

\smallskip

\noindent
\emph{Keywords:} Kre\u\i n space, Kre\u\i n C*-module, tensor product.

\smallskip

\noindent
\emph{MSC-2010:}
					47B50, 
					46C20, 			
					46L08, 
					46L87,			
					46M15, 			
					53C50. 			
\end{abstract}

\tableofcontents



\section{Introduction}

Vector spaces with an indefinite inner product started to appear in physics with the work on special relativistic space-time by H.Minkowski~\cite{M}  and were later used for the first time in quantum field theory by P.Dirac~\cite{D} and W.Pauli~\cite{P}, but their first mathematical discussion was provided by L.Pontrjagin~\cite{Po} and since then they have been an object of study mainly of the Russian school.

Kre\u\i n spaces, i.e.~complete vector spaces equipped with an indefinite inner product, were formally defined by Ju.Ginzburg~\cite{Gi} and in their  present form by E.Scheibe~\cite{Sc}. 
Their properties have been investigated by several mathematicians such as I.Iohvidov, H.Langer, R.Phillips, M.Na\u\i mark, M.Kre\u\i n and his school 
and have been extensively used in quantum field theory via the Gupta-Bleuler~\cite{B,G} formalism in quantum electrodynamics.\footnote{See as main references: J.Bognar\cite{Bo}, T.Azizov, I.Iohvidov~\cite{AI}, M.Dritschel, J.Rovnyak~\cite{DR}, E.Kissin, V.Shulman~\cite{KS}.} 
They have been reconsidered in quantum field theory by K.Kawamura~\cite{K,K2} who also proposed axioms for Kre\u\i n C*-algebras (involutive algebras of bounded linear operators on a Kre\u\i n space). Kre\u\i n spaces also appeared prominently in the definition of semi-Riemannian spectral triples in non-commutative geometry, by A.Strohmaier~\cite{Str}, M.Paschke, A.Sitarz~\cite{PS} and more recently by K.van den Dungen, M.Paschke, A.Rennie~\cite{DPR}.

Hilbert C*-modules (complete modules over a C*-algebra with a C*-algebra-valued positive inner product) are a generalization of Hilbert spaces where the field of complex numbers is substituted by a general C*-algebra. They were first introduced in 1953 by I.Kaplanski~\cite{Kap} in the case of commutative unital \hbox{C*-al}\-gebras. Between 1972 and 1974 W.Paschke~\cite{Pa1,Pa2} and M.Rieffel~\cite{R1,R2,R3} extended the theory to the case of modules over arbitrary C*-algebras and after that the subject grew and spread rapidly. 

The purpose of this paper is to introduce an extremely elementary notion of Kre\u\i n C*-module over a Kre\u\i n C*-algebra, generalizing to the module case the usual decomposability condition of a Kre\u\i n space into its ``positive'' and ``negative'' subspaces, and closely following the definition of 
Kre\u\i n C*-module over a \hbox{C*-al}\-gebra elaborated in S.Kaewumpai~\cite{Ka}. 

In practice, such ``decomposable'' kind of Kre\u\i n modules admit a non-canonical splitting as direct sums of Hilbert C*-modules, with opposite signature that, when we allow the algebra to be a Kre\u\i n C*-algebra, can be chosen to be ``compatible'' with one of the fundamental symmetry automorphisms of the algebra. 

We first recall in section~\ref{sec: k-alg} (a variant of) the definition of Kre\u\i n C*-algebra by K.Kawamura~\cite{K,K2} and, in 
section~\ref{sec: k-mod}, for the benefit of the reader, we reproduce in some detail the key definitions and proofs of the main results on Kre\u\i n 
\hbox{C*-modules} over C*-algebras that were developed in S.Kaewumpai's thesis~\cite{Ka}. 

In section~\ref{sec: k-mod-2}, we further extend the previous definition to cover the case of Kre\u\i n C*-modules over Kre\u\i n C*-algebras. In the subsequent section~\ref{sec: cat} we expand the notion of tensor product of Kre\u\i n C*-modules over C*-algebras, formulated in 
R.Tanadkithirun~\cite{T} in order to cover our more general situation and we discuss some of the properties of the categories of modules and bimodules so obtained. 

Several examples illustrating the scope of the definitions are presented. Of particular interest are the possible applications to the spectral geometry of semi-Riemannian manifolds and their non-commutative counterparts. 

As it can be clearly appreciated by these geometric examples of modules of vector fields over the Clifford algebra of a semi-Riemannian manifold, the notion of Kre\u\i n C*-module that is contained here is very specific and corresponds to the special case of tangent bundles admitting a global decomposition as Whitney orthogonal sums of positive and negative definite Hermitian vector sub-bundles i.e.~semi-Riemannian manifolds that are time-orientable and space-orientable.\footnote{For details on semi-Riemannian geometry we refer to B.O'Neill~\cite{O}.} 
More general notions of Kre\u\i n C*-modules are necessary to deal with cases where such global ``splitting'' is not available, but for now we do not enter this interesting discussion that will very likely also require modifications in the definition of Kre\u\i n C*-algebra.

\section{Kre\u\i n C*-Algebras}\label{sec: k-alg}

The following is a variation of the definition of Kre\u\i n C*-algebra introduced by K.Kawamura~\cite{K}. 

\begin{definition}
A \emph{Kre\u\i n C*-algebra} is an involutive complete complex topological algebra  (i.e.~a complete complex topological vector space with a bilinear continuous product and a continuous involution) $\As$ that admits at least one  \emph{fundamental symmetry} i.e.~an involutive automorphism 
$\alpha:\As\to \As$ with $\alpha\circ \alpha=\iota_\As$ and one Banach algebra norm $\|\cdot\|_\alpha$ (inducing the given topology) such that 
$\|\alpha(a^*)a\|_\alpha=\|a\|_\alpha^2$, for all $a\in \As$.
\end{definition}

\begin{proposition}[K.~Kawamura, Example~2.4, Section~2.3]
The set $\B(K)$ of linear continuous operators on a Kre\u\i n space $K$ is a Kre\u\i n C*-algebra. Every fundamental symmetry $J$ of a Kre\u\i n space $K$ is associated to a fundamental symmetry $a\mapsto JaJ$ , $a\in \B(K)$ of the Kre\u\i n C*-algebra $\B(K)$. 
\end{proposition}

\begin{remark}
Note that although, contrary to K.Kawamura, we assume the existence of a given topology, we do not fix a priori any Banach norm on the Kre\u\i n 
\hbox{C*-al}\-gebra, so that several topologically equivalent Banach norms can exist. Specifically, for every fundamental symmetry $\alpha$ there is a unique norm $\|\cdot\|_\alpha$ making $\As$ a C*-algebra, denoted by $\As^\alpha$, that coincides with $\As$ as a complex algebra and whose involution is given by $x^{\dagger_\alpha}:=\alpha(x^*)$, for all $x\in \As$. 

For example, different fundamental symmetries of a Kre\u\i n space $K$, induce operator norms on the Kre\u\i n C*-algebra $\Bs(K)$ of bounded linear operators on $K$ that do not coincide, although they are topologically equivalent. 
\xqed{\lrcorner} 
\end{remark}

In the subsequent sections, we will often use the notation $\As_+$ for the \emph{even part} of the Kre\u\i n \hbox{C*-algebra} $\As$ under a fundamental symmetry $\alpha$, i.e.~the \hbox{C*-al}\-gebra of elements such that $\alpha(x)=x$; and similarly the notation $\As_-$ for the \emph{odd part} of the Kre\u\i n C*-algebra $\As$ under a fundamental symmetry $\alpha$ i.e.~the Hilbert C*-module, over $\As_+$, of elements such that 
$\alpha(x)=-x$.  

Later on we will see natural situations motivated from semi-Riemannian geometry that seem to require further generalization of the definition of 
Kre\u\i n \hbox{C*-al}\-gebra, but for this work we will mostly limit our consideration to the definition above.

\section{Kre\u\i n C*-modules over C*-algebras}\label{sec: k-mod}

In this section, we recall some basic material on unital Kre\u\i n C*-modules over unital C*-algebras that was developed in S.Kaewumpai's Master thesis~\cite{Ka}. The material naturally covers, as a special case, the situation of Kre\u\i n spaces (that are Kre\u\i n C*-modules over the C*-algebra $\CC$) and will be further generalized in the subsequent section where we will consider modules over Kre\u\i n C*-algebras. 

\medskip 

Recall that, given a unital right module $\Es_\As$ over a unital C*-algebra $\As$, an $\As$-valued Hermitian \emph{inner product} is a map 
$\ip{\cdot}{\cdot}:\Es\times\Es\to\As$ such that, for all $x,y,z\in \Es$, $a\in \As$: 
\begin{gather*}
\ip{z}{x+y}=\ip{z}{x}+\ip{z}{y}, 
\quad 
\ip{z}{xa}=\ip{z}{x}a, 
\quad 
\ip{x}{y}^*=\ip{y}{x}. 
\end{gather*}
In the case of unital left modules ${}_\As\Es$, the second property above is substituted with $\ip{ax}{z}=a\ip{x}{z}$. Whenever confusion might arise, we will denote an inner product on $\Es_\As$ by $\ip{\cdot}{\cdot}_\Es$ and an inner product on ${}_\As\Es$ by ${}_\Es\ip{\cdot}{\cdot}$. 
The direct sum $\Es\oplus \Fs$ of two right (left) unital modules, over the unital C*-algebra $\As$, equipped with the inner product defined by 
$\ip{x_1\oplus y_1}{x_2\oplus y_2}_{\Es\oplus\Fs}:=\ip{x_1}{x_1}_\Es+\ip{y_1}{y_2}_\Fs$, for all $x_1,x_2\in \Es$ and $y_1,y_2\in \Fs$, is a right (left) unital module over $\As$ called \emph{orthogonal direct sum} of $\Es_\As$ and $\Fs_\As$. 

\medskip 

A Hermitian inner product is \emph{positive} if $\ip{x}{x}\in \As_>:=\{a^*a  \ | \ a\in \As \}$, for all $x\in \Es$, where $\As_>$ denotes the positive part of the unital C*-algebra $\As$. An inner product is \emph{non-degenerate} if $\ip{x}{x}=0 \imp x=0_\As$. 
A unital right (left) \emph{Hilbert C*-module} $\Es$ over the unital C*-algebra $\As$ is a unital right (left) module on the unital C*-algebra $\As$, that is equipped with an $\As$-valued positive non-degenerated inner product, making it a complete metric space with respect to the norm 
$\|x\|:=\sqrt{\|\ip{x}{x}\|_\As}$.  

The family $\Ls(\Es_\As)$ of $\As$-linear operators on the right (left) unital $\As$-module $\Es_\As$ is a complex algebra with multiplication given by composition of maps. If the modules $\Es_\As$ and $\Fs_\As$ are equipped with $\As$-valued inner products, we say that a map $T:\Es\to\Fs$ is 
\emph{adjointable} if there exists another map $S:\Fs\to\Es$ such that $\ip{y}{T(x)}_\Fs=\ip{S(y)}{x}_\Es$, for all $x\in \Es$ and $y\in \Fs$. If the inner product is non-degenerate, adjointable maps are necessarily unique and $\As$-linear and, denoting by $T^*$ the unique adjoint of $T$ we have, for Hermitian inner products, $(T^*)^*=T$, $(T\circ S)^*=S^*\circ T^*$ and $(T_1+\alpha T_2)^*=T_1^*+\cj{\alpha}T_2^*$ so that the family 
$\Bs(\Es_\As)$ of adjointable operators is an involutive complex subalgebra of $\Ls(\Es_\As)$. In the case of Hilbert C*-modules, adjointable maps are always continuous and the algebra $\Bs(\Es_\As)$ is always a unital C*-algebra when equipped with the operator norm.\footnote{For all the details of these standard results on Hilbert C*-modules, the reader can consult N.Landsman~\cite{L,L1}, E.Lance~\cite{La}, N.Wegge-Olsen~\cite{WO} or S.Kaewunpai~\cite{Ka}.}  

\begin{definition}
If $\Es_\As$ is module with an inner product $\ip{\cdot}{\cdot}$ over the C*-algebra $\As$, its \emph{antimodule}, here denoted by $-\Es_\As$, 
is the $\As$-module $\Es_\As$ equipped with the opposite inner product. $-\ip{\cdot}{\cdot}$.
\end{definition}

\begin{definition}
A unital right \emph{Kre\u\i n C*-module} $\Ks_\As$ over the unital C*-algebra $\As$ is a unital right module over $\As$ with an $\As$-valued inner product such that $\Ks_\As$ is isomorphic to the ``indefinite'' orthogonal direct sum $\Ks_\As=\Ks_{+\As}\oplus \Ks_{-\As}$ of a Hilbert C*-module $\Ks_{+\As}$, with the antimodule $\Ks_{-\As}$ of a Hilbert C*-module $-\Ks_{-\As}$. Unital left Kre\u\i n C*-modules are similarly defined.  
\end{definition}
Any such decomposition of a Kre\u\i n C*-module over $\As$ will be called a \emph{fundamental decomposition}. 
Fundamental decompositions are in general not unique. 

\begin{definition}
To every fundamental decomposition $\Ks_\As=\Ks_{+\As}\oplus \Ks_{-\As}$ of a right unital Kre\u\i n \hbox{C*-module} $\Ks_\As$ over the unital C*-algebra 
$\As$ there is an associated \emph{fundamental symmetry}  
operator $J: \Ks_\As\to \Ks_\As$ given by 
\begin{equation*}
J (x):=x_+-x_-, \quad \text{where} \quad x=x_+ + x_- 
\end{equation*} 
is the direct sum decomposition of $x\in \Ks_\As$; and there is an associated Hilbert C*-module 
\begin{equation*}
|\Ks|_\As:=\Ks_{+\As}\oplus (-\Ks_{-\As}). 
\end{equation*}

\end{definition}

Making use of the unicity of the expression of vectors as sums of components belonging to the direct summands of a fundamental decomposition and of the definition of the fundamental symmetry associated to such decomposition, we obtain the following statement. 
\begin{proposition}
Let $\Ks_\As = \Ks_{+\As}\oplus \Ks_{-\As}$ be a fundamental decomposition of a Kre\u\i n $\As$-C*-module $\Ks_\As$,  
its fundamental symmetry operator $J: \Ks_\As\to \Ks_\As$ satisfies the following properties for all $x,y\in \Ks$ and $a\in \As$:
\begin{gather*}
J(x+y)=J(x)+J(y), \quad J(xa)=J(x)a, \quad 
J\circ J=\id_\Ks, 
\\
\ip{J(x)}{y}_\Ks=\ip{x}{J(y)}_\Ks, \quad \text{or equivalently}\quad \ip{J(x)}{J(y)}_\Ks=\ip{x}{y}_\Ks, 
\\
\pm\ip{(\id_\Ks\pm J) (x)}{(\id_\Ks\pm J)(x)}_\Ks\geq 0. 
\end{gather*}
Every map $J:\Ks\to\Ks$ that satisfies the previous list of properties is the fundamental symmetry of a unique fundamental decomposition, denoted by \hbox{$\Ks=\Ks^J_+\oplus \Ks^J_-$}, where $\Ks^J_\pm:=\{x\in \Ks \ | \ J(x)=\pm x\}$. The operators $(\id_\Ks\pm J)/2$ are a pair of orthogonal projections onto $\Ks^J_\pm$. 
The inner product in the Kre\u\i n C*-module $\Ks_\As=\Ks^J_{+\As}\oplus \Ks^J_{-\As}$ and the inner product in the associated Hilbert C*-module 
$|\Ks|^J_\As:=\Ks^J_{+\As}\oplus (-\Ks^J_{-\As})$ are related by 
\begin{equation*}\label{eq: k-k}
\ip{x}{y}_\Ks=\ip{J(x)}{y}_{|\Ks|^J}, \quad \ip{J(x)}{y}_\Ks=\ip{x}{y}_{|\Ks|^J}, \quad \forall x,y\in \Ks.  
\end{equation*}
\end{proposition}

\begin{theorem} \label{th: iso}
If $\Ks_\As$ is a unital right (left) Kre\u\i n C*-module over the unital C*-algebra $\As$ and $J_1,J_2$ are two fundamental symmetries of $\Ks$ with associated fundamental decompositions \hbox{$\Ks=\Ks^{J_1}_{+}\oplus \Ks^{J_1}_{-}$} and $\Ks=\Ks^{J_2}_{+}\oplus \Ks^{J_2}_{-}$, there are two 
$\As$-linear continuous bijective maps of Hilbert $\As$-C*-modules \hbox{$T^{J_2J_1}_{+}:\Ks^{J_1}_{+}\to \Ks^{J_2}_{+}$} and 
$T^{J_2J_1}_{-}:-(\Ks^{J_1}_{-})\to -(\Ks^{J_1}_{-})$, given by $T^{J_2J_1}_{+}(x^{J_1}_{+}):=x^{J_2}_{+}$ and 
$T^{J_2J_1}_{-}(x^{J_1}_{-}):=x^{J_2}_{-}$, where 
$x=x^{J_1}_{+} +x^{J_1}_{-}$ and $x=x^{J_2}_{+} + x^{J_2}_{-}$ are the direct sum decompositions of $x\in \Ks_\As$ in each of the fundamental decompositions.  
\end{theorem}
\begin{proof}
The map $T^{J_2J_1}_{+}: \Ks^{J_1}_{+}\to \Ks^{J_2}_{+}$ is adjointable with adjoint given by the map 
\begin{gather*}
T^{J_1J_2}_{+}: \Ks^{J_2}_{+}\to \Ks^{J_1}_{+}, \quad \text{defined by} \quad T^{J_1J_2}_+(y):=y^{J_1}_{+}, 
\quad \text{for $y\in \Ks^{J_2}_{+}$}, 
\end{gather*}
since $\ip{T^{J_2J_1}_{+}(x)}{y}_{\Ks^{J_2}_+} = \ip{x^{J_2}_{+}}{y^{J_2}_{+}}_{\Ks^{J_2}_+}=\ip{x}{y}_{\Ks}
=\ip{x^{J_1}_+}{y}_{\Ks^{J_1}_+}=\mip{x}{T^{J_1J_2}_{+}(y)}_{\Ks^{J_1}_+}$, for all $x\in \Ks^{J_1}_+$ and all $y\in \Ks^{J_2}_+$. 
In exactly the same way we have that $T^{J_2J_1}_{-}:\Ks^{J_1}_{-}\to \Ks^{J_2}_{-}$ is adjointable with adjoint given by the map 
$T^{J_1J_2}_{-}: \Ks^{J_2}_{-}\to \Ks^{J_1}_{-}$ defined by $T^{J_1J_2}_{-}(y):=y^{J_1}_{-}$, for all $y\in \Ks^{J_2}_{-}$. 
Since an adjointable operator between Hilbert C*-modules (and also between anti-Hilbert C*-modules) is necessarily continuous, both 
$T^{J_2J_1}_{+}, T^{J_2J_1}_{-}$ and their adjoints $T^{J_1J_2}_{+}, T^{J_1J_2}_{-}$ are continuous $\As$-linear maps. 

If $x,y\in \Ks^{J_1}_{+}$ and $T^{J_2J_1}_{+}(x)=T^{J_2J_1}_{+}(y)$, we have $x^{J_1}_+ - y^{J_1}_+=x-y=x^{J_2}_- - y^{J_2}_-$ and since we have $0\leq\ip{x-y}{x-y}_{\Ks^{J_1}_+}=\ip{x-y}{x-y}_\Ks=\ip{x^{J_2}_- - y^{J_2}_-}{x^{J_2}_- - y^{J_2}_-}_{\Ks^{J_2}_-}\leq 0$, via the non-degeneracy of the Kre\u\i n C*-module $\Ks_\As$, we obtain the injectivity of $T^{J_2J_1}_+$. 

Let $(x_n)$ be a sequence in $\Ks^{J_1}_+$ such that $T^{J_2J_1}_{+}(x_n)$ converges to a given point 
\hbox{$z\in \cj{\im(T^{J_2J_1}_{+})}\subset \Ks^{J_2}_{+}$}. Since $\ip{(x_n-x_m)^{J_2}_{-}}{(x_n-x_m)^{J_2}_{-}}_\Ks\leq 0$, we have
\begin{align*}
\|&x_n-x_m\|^2_{\Ks^{J_1}_+}=\ip{x_n-x_m}{x_n-x_m}_\Ks \\
&\leq \ip{T^{J_2J_1}_{+}(x_n-x_m)}{T^{J_2J_1}_{+}(x_n-x_m)}_\Ks=\|T^{J_2J_1}_{+}(x_n)-T^{J_2J_1}_{+}(x_m)\|^2_{\Ks^{J_2}_+}, 
\end{align*}
so that $(x_n)$, being a Cauchy sequence in the Hilbert C*-module $\Ks^{J_1}_{+}$, converges to a point $x_o\in \Ks^{J_1}_{+}$. 
By continuity of $T^{J_2J_1}_+$ we get $z=\lim_{n\to \infty}T^{J_2J_1}_{+}(x_n)=T^{J_2J_1}_{+}(x_o)\in \im(T^{J_1J_2}_+)$, that provides the closure of the range of $T^{J_2J_1}_+$.  

The fact that $T^{J_2J_1}_+:\Ks^{J_1}_+\to\Ks^{J_2}_+$ is an adjointable map between Hilbert C*-modules with closed range is equivalent (see for example N.Wegge-Olsen~\cite[corollary 15.3.9]{WO}) to the complementability of the submodule $\im(T^{J_2J_1}_+)\subset \Ks^{J_2}_+$.   

If $y\in \im(T^{J_2J_1}_{+})^\perp\subset \Ks^{J_2}_+$, 
for all $x\in \Ks^{J_1}_{+}$, we get $\ip{T^{J_1J_2}_{+}(y)}{x}_{\Ks^{J_1}_+}=\ip{y}{T^{J_2J_1}_{+}(x)}_{\Ks^{J_2}_+}=0$ 
and hence $y\in \ke(T^{J_1J_2}_{+})$. Since $T^{J_1J_2}_{+}$ is injective, we have $\im(T^{J_2J_1}_{+})^\perp=\{0\}$ and since 
$\im(T^{J_2J_1}_{+})$ is complementable, from $\Ks^{J_2}_+=\im(T^{J_2J_1}_{+})\oplus \im(T^{J_2J_1}_{+})^\perp=\im(T^{J_2J_1}_{+})$, we obtain the surjectivity of $T^{J_2J_1}_+$. 
\end{proof}

\begin{theorem}
If $\Ks_\As$ is a unital right (left) Kre\u\i n C*-module, over $\As$, with two fundamental symmetries $J_1$ and $J_2$ with associated direct sum decompositions $\Ks=\Ks^{J_1}_{+}\oplus \Ks^{J_1}_{-}=\Ks^{J_2}_{+}\oplus \Ks^{J_2}_{-}$, then the Hilbert C*-modules 
$|\Ks|^{J_1}:=\Ks^{J_1}_{+}\oplus \Ks^{J_1}_{-}$ 
and $|\Ks|^{J_2}:=\Ks^{J_2}_{+}\oplus \Ks^{J_2}_{-}$ 
have equivalent norms. Hence on $\Ks_\As$ there is a natural topology called the \emph{strong topology}.
\end{theorem}
\begin{proof}
By theorem~\ref{th: iso}, we have two bijective $\As$-linear adjointable (hence continuous) maps 
\begin{gather*}
T^{J_2J_1}_{+}: \Ks^{J_1}_{+}\to \Ks^{J_2}_{+}, \quad 
T^{J_2J_1}_{-}: -\Ks^{J_1}_{-}\to -\Ks^{J_2}_{-}.
\end{gather*}

It follows immediately that their direct sum map 
$T^{J_2J_1}_{+}\oplus T^{J_2J_1}_{-}: |\Ks|^{J_1}\to |\Ks|^{J_2}$  
is a bijective linear continuous map with continuous inverse and hence as Banach spaces $|\Ks|^{J_1}$ and $|\Ks|^{J_2}$ have equivalent norms.  
\end{proof}

\begin{proposition}\label{proposition: a-k}
Let $\Ks_A=\Ks^J_{+}\oplus \Ks^J_{-}$ be the fundamental decomposition of the right (left) unital Kre\u\i n C*-module $\Ks_\As$ associated to the fundamental symmetry $J$ and with associated Hilbert C*-module $|\Ks|^J_\As$. 
The operator $T: \Ks_\As\to \Ks_\As$ is an adjointable operator in $\Ks_\As$ if and only if it is adjointable in $|K|^J_\As$. 
The adjoint $T^*$ of $T$ in the Kre\u\i n C*-module $\Ks_\As$ and the adjoint $T^{\dag_J}$ of $T$ in the Hilbert C*-module $|\Ks|^J_\As$ are related by the following formulas:
\begin{equation*}
T^{\dag_J}=J\circ T^*\circ J, \quad T^*=J\circ T^{\dag_J}\circ J.
\end{equation*}

The family $\Bs(\Ks_\As)$ of adjointable operators in the unital right (left) Kre\u\i n \hbox{C*-module} $\Ks_\As$ coincides as a set with the C*-algebra 
$\Bs(|\Ks|^J_\As)$ of adjointable operators in the unital right (left) Hilbert C*-module $|\Ks|_\As$.
\end{proposition}
\begin{proof}
If $T$ is adjointable in the Kre\u\i n C*-module $\Ks_\As$ with adjoint $T^*$, we have, for all $x,y\in \Ks_\As$,
$\ip{x}{T(y)}_{\Ks}=\ip{T^*(x)}{y}_{\Ks}$ and so $\ip{J(x)}{T(y)}_{|\Ks|^J}=\ip{J(T^*(x))}{y}_{|\Ks|^J}$ and 
taking $x:=J(z)$ we get 
$\ip{z}{T(y)}_{|\Ks|^J}=\ip{J\circ T^*\circ J(z)}{y}_{|\Ks|^J}$ that gives the adjointability of $T$ in $|\Ks|^J_\As$ with adjoint 
$T^{\dag_J}=J\circ T^*\circ J$. 
Following the same passages in the reverse order establishes the equivalence of the notions of adjointability in $\Ks_\As$ and in $|\Ks|^J_\As$. 
\end{proof}

\begin{theorem}
The algebra $\Bs(\Ks_\As)$ of adjointable endomorphisms of a unital right (left) Kre\u\i n \hbox{C*-module} $\Ks_\As$ is a Kre\u\i n-C*-algebra.  
\end{theorem}
\begin{proof}
By the previous proposition, we see that $\Bs(\Ks_\As)=\Bs(|\Ks|^J_\As)$ as sets and also as unital associative algebras, since the operations of addition and scalar multiplication are the same. Taking on $\Bs(\Ks_\As)$ the operator norm $\|\cdot\|_J$ defined in the C*-algebra $\Bs(|\Ks|^J_\As)$ of the unital Hilbert C*-module $|\Ks|^J_\As$, we see that $\Bs(\Ks_\As)$ is a Banach space. The involution on the algebra $\Bs(\Ks_\As)$ is obtained via the Kre\u\i n C*-module adjoint $T\mapsto T^*$. In order to complete the proof that $\Bs(\Ks_\As)$ is a Kre\u\i n C*-algebra, we need to provide an involutive automorphism $\alpha: \Bs(\Ks_\As)\to \Bs(\Ks_\As)$ such that $\|\alpha(T^*)\circ T\|_J=\|T\|_J^2$ for all $T\in \Bs(\Ks_\As)$.
For this purpose, for all $T\in\Bs(\Ks_\As)$, we define $\alpha(T):=J\circ T\circ J$ and, using the C*-property of $\Bs(|\Ks|^J_\As)$, 
verify that $\|\alpha(T^*)\circ T\|_J=\|(J\circ T\circ J) \circ T\|_J=\|T^{\dag_J}\circ T\|_J=\|T\|_J^2$. 
\end{proof}

\begin{proposition}
Let $\Ks_\As$ be a unital right (left) Kre\u\i n C*-module. Any two fundamental symmetries $J_1,J_2\in \Bs(\Ks_\As)$ of $\Ks$ are unitarily equivalent.  
\end{proposition}
\begin{proof}
With the notation used in theorem~\ref{th: iso}, consider the adjointable operator \hbox{$U:=T^{J_2J_1}_+\oplus T^{J_2J_1}_-$} and note that 
$U\circ J_1=J_2\circ U$ with $U^*=U^{-1}$. 
\end{proof}

\begin{example}
Every Kre\u\i n-space is a Kre\u\i n-C*-module over the C*-algebra $\CC$, the complexification $\MM_\CC:=\MM\otimes_\RR\CC$ of Minkowski space in special relativity being probably the most important example. 
\xqed{\lrcorner}
\end{example}

\begin{example}\label{ex: man}
Let $M$ be a semi-Riemannian manifold that for now (although this makes the situation less interesting for applications to physics) we suppose to be compact. If the manifold is also supposed to be space-orientable and time-orientable (for a specific example, consider $\TT^2$ with the indefinite metric coming from the product of a copy of $\TT$ with positive metric and another copy with negative metric), the module $\Gamma(T(M))$ of continuous sections of its tangent bundle $T(M)$ is a unital Kre\u\i n C*-module over the unital C*-algebra $C(M)$ of continuous functions. 

Note that, when a semi-Riemannian manifold is not space-orientable and time-orientable, although each fiber of the tangent bundle $T(M)$ is still a 
Kre\u\i n space, the module $\Gamma(T(M))$ of continuous vector fields, fails to be a Kre\u\i n C*-module over the C*-algebra $C(M)$, because in this case $\Gamma(T(M))$ does not admit a global splitting as a direct sum of a Hilbert and an anti-Hilbert C*-module 
(equivalently the tangent bundle $T(M)$ is not a Whitney direct sum of a positive definite and a negative definite sub-bundle). 

As a consequence, this very interesting kind of complete semi-definite Hilbert C*-modules do not admit a globally defined fundamental symmetry! 
\xqed{\lrcorner} 
\end{example}

\section{Kre\u\i n C*-modules over Kre\u\i n C*-algebras}\label{sec: k-mod-2}

\begin{definition}
A \textbf{left Kre\u\i n C*-module} over a Kre\u\i n C*-algebra $\As$ is a complete topological vector space that is also a left (unital) module $\Ks$ over $\As$ with the the following properties: 
\begin{itemize}
\item
$\Ks$ is equipped with an $\As$-valued inner product ${}_\As\mip{\cdot}{\cdot}:\Ks\times\Ks\to\As$ such that
\begin{gather*}
{}_\As\mip{x+y}{z}={}_\As\mip{x}{z}+{}_\As\mip{y}{z}, \quad \forall x,y,z\in\Ks, 
\\
{} _\As\mip{a x}{y}=a\cdot {}_\As\mip{x}{y},\quad \forall x,y\in\Ks,\ \forall a\in\As, 
\\
{}_\As\mip{x}{y}^*={}_\As\mip{y}{x},\quad \forall x,y\in\Ks, 
\\
\forall y\in \Ks, \ {}_\As\mip{x}{y}=0 \Rightarrow x=0,  
\end{gather*}	
\item
there exists a \emph{fundamental symmetry}
$J_\As:\Ks\to\Ks$ such that $J\circ J=\id_\Ks$, 
\begin{gather*}
J_\As(x+y)=J_\As(x)+J_\As(y), \quad \forall x,y\in\Ks, 
\\ 
J_\As(a \cdot x)=\alpha(a)\cdot J_\As(x), \quad \forall x\in \Ks, \ \forall a\in \As, 
\\
\alpha({}_\As\mip{x}{y})={}_\As\mip{J_\As(x)}{J_\As(y)}, \quad \forall x,y\in\Ks, 
\end{gather*} 	
\item 
for the given choice of fundamental symmetries $\alpha$ on $\As$ and $J_\As$ on $\Ks$, we have that the map $(x,y)\mapsto {}_\As\mip{x}{J_\As(y)}$ gives to $\Ks$ the structure of a Hilbert C*-module over the C*-algebra $\As^\alpha$ whose norm induces the original topology of $\Ks$.  
\end{itemize}
\end{definition}
In a perfectly similar way, there is a definition of \emph{right Kre\u\i n C*-module $\Ks$} over a Kre\u\i n C*-algebra $\Bs$, where the $\Bs$-valued inner product $\mip{\cdot}{\cdot}_\Bs:\Ks\times\Ks\to\Bs$ satisfies $\mip{x}{yb}_\Bs=\mip{x}{y}_\Bs \cdot b$, for all $b\in \Bs$ and all $x,y\in \Ks$; the fundamental symmetry $J_\Bs$ satisfies $J_\Bs(x\cdot b)=J_\Bs(x)\cdot \beta(b)$, for all $x\in \Ks$ and $b\in \Bs$; and where all the other properties remain essentially unchanged. 

\medskip 

The following definition of bimodule, whenever we further impose the additional fullness requirements 
$\As={}_\As\mip{\Ks}{\Ks}:=\cj{\Span\{{}_\As\mip{x}{y}\ | \ x,y\in\Ks\}}$ and 
$\Bs=\mip{\Ks}{\Ks}_\Bs:=\cj{\Span\{\mip{x}{y}_\Bs \ | \ x,y\in \Ks\}}$, extends to the Kre\u\i n C*-algebras context the usual notion of imprimitivity Hilbert C*-bimodule that provides the well-known (strong) Morita equivalence of C*-algebras (see also remark~\ref{rem: morita}). 
\begin{definition}
A \textbf{Kre\u\i n C*-bimodule} ${}_\As\Ks_\Bs$ is a left Kre\u\i n C*-module over $\As$ that is the same time a right Kre\u\i n C*-module over 
$\Bs$ with additional properties: $(a\cdot x)\cdot b=a\cdot(x\cdot b)$, $J_\As=J_\Bs$ and with right and left inner products related via: 
${}_\As\mip{x}{y}x=x\mip{y}{x}_\Bs$, for all $x,y\in\Ks$.  
\end{definition}

The compatibility condition requested above on the inner products assures that the induced left and right norms on the Hilbert C*-module $\Ks$ coincide. 

\begin{remark}
In our definitions the two auxiliary $\As^\alpha$-valued Hilbert C*-module inner products 
\begin{gather*}
\mip{x}{y}^{J_\As}_{\As^\alpha}:=\mip{x}{J_\As(y)}_\As, \quad {}^{J_\As}\mip{x}{y}_{\As^\alpha}:=\mip{J_\As(x)}{y}_\As, 
\end{gather*}
can be used in place of each other since they are related by the isomorphism $\alpha$ of the C*-algebra $\As^\alpha$: 
$\alpha(\mip{J_\As(x)}{y}_\As)=\mip{x}{J_\As(y)}_\As$. 
\xqed{\lrcorner} 
\end{remark}

\begin{remark}
Note that $\Ks_\Bs$ becomes naturally a Kre\u\i n C*-module over the C*-algebra $\Bs_+$ when equipped with the even part of the original inner product  i.e.~taking $\frac{1}{2} \mip{x}{y}_\Bs+\frac{1}{2}\mip{J(x)}{J(y)}_\Bs$ as inner product on $\Ks$; however with this new inner product the submodules  $\Ks_+$ and $\Ks_-$ are orthogonal (as in our original definition of Kre\u\i n C*-module over a C*-algebra) contrary to the general situation here, where the obstruction to the orthogonality is measured by the odd part of the original inner product 
$\frac{1}{2}\mip{J(x)}{y}_\Bs+\frac{1}{2}\mip{x}{J(y)}_\Bs$ that is always a non-degenerate anti-Hermitian product on $\Ks$ with values in 
$\Bs_-$. Furthermore, although in some situations (for example in finite dimensional cases) the product topology induced by the decomposition 
$\Ks=\Ks_+\oplus\Ks_-$ (that is actually the topology induced by the even part of the inner product, since the two inner products coincide, modulo sign, on their restriction to $\Ks_+$ and $\Ks_-$ respectively) is the same as the original topology of $\Ks$, we still suspect that in general this might fail. 

A significant difference from the case of Kre\u\i n C*-modules over C*-algebras is that the fundamental symmetry $J_\Bs$ is in general not self-adjoint (and in general not even adjointable, as can be seen in the case of example~\ref{ex: alg} below, where adjointability happens only in the case of 
$\alpha$ equal to the identity i.e.~when $\As$ is already a C*-algebra) with respect to the original Kre\u\i n inner product as suggested by the general lack of orthogonality between the even and odd submodules $\Ks_+$ and $\Ks_-$. 
\xqed{\lrcorner} 
\end{remark}

\begin{example}
Every Kre\u\i n C*-module $\Ks$ over a C*-algebra $\As$, as defined in the previous section, is a special case of our new definition as results by taking 
$\alpha$ to be the identity isomorphism of $\As$. Actually, whenever the Kre\u\i n C*-algebra $\As$ is a C*-algebra and $\alpha$ is trivial, we reduce to the definition of the last section: the submodules $\Ks_+$ and $\Ks_-$ are orthogonal and the fundamental symmetry $J_\As$ is Hermitian. The new definition of Kre\u\i n module over a Kre\u\i n C*-algebra even allows for a possible choice of a nontrivial $\alpha$ also in the case of a C*-algebra $\As$ and in this situation we obtain a Kre\u\i n C*-module over a C*-algebra where the two submodules $\Ks_+$ and $\Ks_-$ fail to be orthogonal and $J_\As$ is not Hermitian.  
\xqed{\lrcorner} 
\end{example}

\begin{example}
Let $\Ks_\As$ be a right Kre\u\i n C*-module over the C*-algebra $\As$ and let $\Bs(\Ks_\As)$ be the Krein C*-algebra of adjointable operators on $\Ks$, 
then ${}_{\Bs(\Ks)}\Ks$ is a left Kre\u\i n C*-module over the Kre\u\i n C*-algebra $\Bs(\Ks_\As)$ with left inner product defined by 
${}_{\Bs(\Ks)}\mip{x}{y}:=\Theta_{x,y}$, where $\Theta_{x,y}(z):=x\cdot \mip{y}{z}_\As$, for all $x,y,z\in \Ks$. 

The bimodule ${}_{\Bs(\Ks)}\Ks_\As$ is actually a Kre\u\i n C*-bimodule such that ${}_{\Bs(\K)}\mip{x}{y}z=x\mip{y}{z}_\As$. for all $x,y,z\in\Ks$. 
\xqed{\lrcorner} 
\end{example}

\begin{example}\label{ex: alg}
Let $\As$ be a Kre\u\i n C*-algebra. Then ${}_\As\As$ and $\As_\As$ are both Kre\u\i n C*-modules with inner products given by 
$\mip{x}{y}_\As:=x^*y$ and ${}_\As\mip{x}{y}:=xy^*$, furthermore ${}_\As\As_\As$ is a Kre\u\i n C*-bimodule.
\xqed{\lrcorner} 
\end{example}

\begin{example}
Let $K_{1}$ and $K_{2}$ be two Kre\u\i n spaces. 

The space ${}_{\B(K_2)}\B(K_1,K_2)_{\B(K_1)}$ of linear continuous maps between them is a Kre\u\i n C*-bimodule with the left/right actions given by the usual compositions of linear operators and inner products given respectively by 
$\mip{T}{S}_{\Bs(K_1)}:=T^*\circ S$ and ${}_{\Bs(K_2)}\mip{T}{S}:=T\circ S^*$. 
\xqed{\lrcorner} 
\end{example}

\begin{example}
Following the definitions provided in~\cite{BR}, let $A,B\in \Ob_{\Af}$ be two objects in a Kre\u\i n \hbox{C*-category} $\Af$. 
Then $\Af_{AB}:=\Hom_{\Af}(B,A)$ is a Kre\u\i n \hbox{C*-bimodule} over the Kre\u\i n C*-algebras $\Af_{AA}$ on the left and $\Af_{BB}$ on the right.
\xqed{\lrcorner} 
\end{example}

\begin{example}
Let $\MM$ be Minkowski space (or more generally any real vector space equipped with semi-definite inner product); 
let $\Lambda^{\CC}(\MM)$ denote the space of complex-valued antisymmetric forms on $\MM$ (the complexified Grassmann algebra of $\MM$) and let $\CCl(\MM)$ denote the complexified Clifford algebra of $\MM$. 

Note that for every fundamental decomposition of $\MM=\MM_+\oplus\MM_-$, we have for the Grassmann algebras the decomposition 
$\Lambda^{\CC}(\MM):=\Lambda^{\CC}(\MM_+)\hat{\otimes}\Lambda^{\CC}(\MM_-)$ and similarly, for the Clifford algebras, 
$\CCl(M)=\CCl(\MM_+)\hat{\otimes}\CCl(\MM_-)$, where (if we work in the category of associative algebras) $\hat{\otimes}$ denotes the $\ZZ_2$-graded tensor product. 

Note that the (underlying complex vector space of the) Grassmann algebra $\Lambda^\CC(\MM)$ is naturally a Kre\u\i n space with the semi-definite inner product induced by universal factorization property via 
$\mip{\omega_1\wedge\cdots\wedge\omega_n}{\phi_1\wedge\cdots\phi_n}:=\det[\mip{\omega_i}{\phi_j}]$ on each of the summands in  
$\oplus_{q=0}^{\dim\MM}\Lambda^\CC_q(\MM)= \Lambda^\CC(\MM)$, where $\mip{\omega_i}{\phi_j}\in \CC$ denotes the Minkowski inner product on complexified covectors $\omega_i,\phi_j\in \Lambda^\CC_1(\MM)$.   

The Clifford algebra $\CCl(\MM)$ has a natural structure of Ke\u\i n C*-algebra as a sub-algebra of the Kre\u\i n C*-algebra $\B(\Lambda^\CC(\MM))$: 
every fundamental symmetry $J_\MM$ of Minkowski space lifts to an involutive automorphims of $\CCl(\MM)$ (by the defining universal property of the Clifford algebra) that, under the linear isomorphism $\CCl(\MM)\simeq\Lambda^\CC(\MM)$, coincides with the (second quantized) fundamental symmetry $J_{\Lambda^\CC(\MM)}:=\oplus^{\infty}_{q=0}J_\MM^{\wedge q}$ of the Kre\u\i n space $\Lambda^\CC(\MM)$. 
It follows that the Kre\u\i n space $\Lambda^\CC(\MM)$ is a left Kre\u\i n module over the Kre\u\i n C*-algebra $\CCl(\MM)$. 

The (underlying vector space of the) Grassmann algebra $\Lambda^\CC(\MM)$ also becomes a Kre\u\i n C*-bimodule over $\CCl(\MM)$ via Clifford left and right actions and with the inner products induced via the linear isomorphism $\Lambda^\CC(\MM)\simeq\CCl(\MM)$ by the standard Kre\u\i n 
C*-bimodule structure of the Kre\u\i n C*-algebra $\CCl(\MM)$ over itself. 

As described in more detail in H.Baum~\cite{Ba} (see also A.Strohmaier~\cite[section~5.1]{Str} 
and K.Van Den Dungen-M.Paschke-A.Rennie~\cite[section~3.3.1]{DPR}) the module $S(\MM)$ of (Dirac) spinors is a Kre\u\i n space, whose fundamental symmetries are proportional to the product of the operators of Clifford multiplication by all the vectors in an orthonormal basis for the timelike summand of a fundamental decomposition $\MM=\MM_+\oplus \MM_-$.\footnote{On the usual Minkowski space $\MM^4$, the Kre\u\i n space 
$S(\MM^4)$ has signature $(2,2)$ and the fundamental symmetries are just the Dirac $\gamma^0$ operators.} 
The space $S(\MM)$ (for $\MM$ even-dimensional) becomes a left Kre\u\i n \hbox{C*-module} over the Ke\u\i n C*-algebra $\CCl(\MM)$ with the inner product induced by the linear isomorphisms $S(\MM)\otimes S(\MM)^*\simeq \Lambda^\CC(\MM)\simeq\CCl(\MM)$ and ${}_{\CCl(\MM)}S(\MM)_\CC$ is a Kre\u\i n C*-bimodule that, with the terminology introduced in remark~\ref{rem: morita}, is an example of Morita-Kre\u\i n equivalence C*-bimodule. 
\xqed{\lrcorner} 
\end{example}

\begin{example}
Let $M$ be a (compact) semi-Riemannian space-orientable and time-orientable manifold. As already described in example~\ref{ex: man}, the module $\Gamma(T(M))$ of its continuous vector fields is a (unital) Kre\u\i n C*-bimodule over the (unital) C*-algebra $C(M)$. 
The algebra $\Gamma(\CCl (M))$ of continuous section of the complexified Clifford bundle $\CCl(T(M))$ of $M$ is a (unital) Kre\u\i n C*-algebra and the module $\Gamma(\Lambda^{\CC}(M))$ of continuous sections of the complexified Grassmann bundle $\Lambda^\CC(M)$ of $M$ is a (unital) Kre\u\i n C*-bimodule over the Kre\u\i n C*-algebra $\Gamma(\CCl(M))$. The case of spinorial manifolds is described in example~\ref{ex: spin}. 
\xqed{\lrcorner} 
\end{example}

\begin{remark}
Note that, in the previous example, if the manifold $M$ is not time-orientable and space-orientable, the algebra $\Gamma(\CCl(M))$ (although being a nice involutive complete topological algebra) does not admit a globally defined fundamental symmetry and so does not fit into the current definition of Kre\u\i n C*-algebra! 

This clearly indicates that the environment of Kre\u\i n C*-algebras and Kre\u\i n C*-modules that we have developed here is insufficient to deal with a general axiomatization of ``complete semi-definite C*-algebras and C*-modules over them''.  
\xqed{\lrcorner} 
\end{remark}

\medskip 

We pass now to briefly examine the main properties of the algebras of adjointable operators on Kre\u\i n C*-modules over Kre\u\i n C*-algebras. 

\begin{definition}
Let $\Ks_\As$ be a Kre\u\i n C*-module over the Kre\u\i n C*-algebra $\As$. A map $T:\Ks\to\Ks$ is said to be \emph{adjointable} if there exists another map $T^*:\Ks\to\Ks$ such that $\mip{T(x)}{y}_\As=\mip{x}{T^*(y)}_\As$, for all $x,y\in \Ks$. The family of such adjointable maps is denoted by 
$\Bs(\Ks_\As)$.  
\end{definition}

\begin{remark}
As usual, the adjointable maps are already $\As$-linear and continuous and the adjoint $T^*$ is unique. The set $\Bs(\Ks_\As)$ is a vector space and an associative unital algebra under composition, furthermore the map $*:T\mapsto T^*$ is involutive, antimultiplicative and conjugate $\CC$-linear so that  $\Bs(\Ks_\As)$ is a complex associative unital $*$-algebra. 
\xqed{\lrcorner} 
\end{remark}

\begin{proposition}
A map $T:\Ks_\As\to\Ks_\As$ is adjointable with respect to the inner product $\mip{\cdot}{\cdot}_\As$ if and only if the map $T$ is adjointable for the Hilbert C*-module $\Ks_{\As^\alpha}$ with the auxiliary inner product $\mip{\cdot}{\cdot}^{J_\As}_{\As^\alpha}$. As a consequence, the associative unital algebra $\Bs(\Ks_\As)$ coincides with the associative unital algebra $\Bs(\Ks_{\As^\alpha})$. The relation between the adjoint $T^*$ of $T$ in 
$\Bs(\Ks_\As)$ and the adjoint $T^{\dag_\alpha}$ of $T$ in the C*-algebra $\Bs(\Ks_{\As^\alpha})$ is given by: 
\begin{equation*}
T^{\dag_\alpha}= J_\As \circ T^* \circ J_\As, \quad T^*=J_\As\circ T^{\dag_\alpha}\circ J_\As
\end{equation*}
\end{proposition}
\begin{proof}
Suppose that $\mip{T(x)}{y}_\As=\mip{y}{T^*(y)}_\As$. The following calculation
\begin{align*}
\mip{T(x)}{y}^{J_\As}_{\As^\alpha}&=\mip{T(x)}{J_\As(y)}_\As=
\mip{x}{T^*J_\As(y)}_\As\\
&=\mip{x}{J_\As J_\As T^*J_\As(y)}_\As=\mip{x}{J_\As T^*J_\As(y)}^{J_\As}_{\As^\alpha},
\end{align*} 
assures that the adjointability in $\Bs(\Ks_\As)$ implies the adjointability in $\Bs(\Ks_{\As^\alpha})$ and the first formula relating the adjoints. 

Suppose now that 
$\mip{T(x)}{y}^{J_\As}_{\As^\alpha}=\mip{x}{T^{\dag_\alpha}(y)}^{J_\As}_{\As^\alpha}$ 
i.e.~$\mip{T(x)}{J_\As(y)}_{\As}=\mip{x}{J_\As T^{\dag_\alpha}(y)}_{\As}$ and choosing $y=J_\As(y')$ for an arbitrary $y'\in \Ks$, we obtain 
$\mip{T(x)}{(y')}_{\As}=\mip{x}{J_\As T^{\dag_\alpha}J_\As(y')}_{\As}$ that assures the reverse and the second adjointability formula. 
\end{proof}

Although we know that in general $J_\As$ is not an adjointable operator, we still have the following fundamental symmetry of $\Bs(\Ks_\As)$: 
\begin{proposition}
If $T$ is adjointable in $\Bs(\Ks_\As)$, also the new operator $J_\As\circ T \circ J_\As$ is adjointable in $\Bs(\Ks_\As)$ and the map 
\begin{equation*}
\alpha_{J_\As}: T\mapsto J_\As\circ T \circ J_\As
\end{equation*}
is a $*$-isomorphism of the involutive algebra $\Bs(\Ks_\As)$ of adjointable operators. 
\end{proposition}
\begin{proof}
Since, for all $T\in \Bs(\Ks_\As)$, $(T^*)^*=T$, we obtain \hbox{$J_\As(J_\As T^{\dag_\alpha}J_\As)^{\dag_\alpha}J_\As=T$} or equivalently 
$(J_\As T^{\dag_\alpha}J_\As)^{\dag_\alpha}=J_\As T J_\As$. 
Since $(S^{\dag_\alpha})^{\dag_\alpha}=S$, we get 
\hbox{$(J_\As T^{\dag_\alpha}J_\As)=(J_\As T J_\As)^{\dag_\alpha}$} i.e.~$\alpha_{J_\As}$ is a \hbox{$\dag_\alpha$-iso}\-mor\-phism of the 
C*-algebra $\Bs(\Ks_{\As^\alpha})$. 

Similarly from $(S^{\dag_\alpha})^{\dag_\alpha}=S$, we get $J_\As(J_\As S^* J_\As)^*J_\As=S$ or equivalently $(J_\As S^* J_\As)^*=J_\As S J_\As$ and hence $J_\As S^* J_\As=(J_\As S J_\As)^*$ i.e.~$\alpha_{J_\As}$ is a $*$-iso\-mor\-phism of the involutive algebra $\Bs(\Ks_\As)$. 
\end{proof}

\begin{theorem}
The algebra $\Bs(\Ks_\As)$ of adjointable operators of a Kre\u\i n \hbox{C*-mod}\-ule over a Kre\u\i n C*-algebra $\As$ is a Kre\u\i n C*-algebra.  
\end{theorem}
\begin{proof}
The $*$-isomorphism $\alpha_{J_\As}$ of the involutive algebra $\Bs(\Ks_\As)$, defined in the previous proposition, satisfies the C*-property 
$\|\alpha_{J_\As}(T^*)T\|_\alpha=\|T\|^2_\alpha$ with respect to the norm of the C*-algebra $\Bs(\Ks_{\As^\alpha})$. 
\end{proof}

\section{Categories of Kre\u\i n C*-modules}\label{sec: cat}

The following proposition, whose proof is self-evident, provides the most elementary category of morphisms of Kre\u\i n C*-algebras that naturally contains, as a full subcategory, the category of unital $*$-homomorphims of unital C*-algebras. 
\begin{proposition}
There is a category $\Af$ whose objects are unital Kre\u\i n \hbox{C*-al}\-gebras $\As,\Bs,\dots$; whose arrows $\phi:\As\to\Bs$ are unital Kre\u\i n 
$*$-homomorphisms i.e.~unital $*$-homomorphisms of involutive unital algebras $\phi:\As\to\Bs$ such that there exist at least a fundamental  symmetry of $\alpha$ of $\As$ and $\beta$ of $\Bs$ such that $\phi\circ\alpha=\beta\circ \phi$; and composition is the usual composition of functions.  
\end{proposition}

We now define the Kre\u\i n C*-analogue of the well-known notion of C*-cor\-re\-spon\-dence. 
\begin{definition}
A \emph{left Kre\u\i n C*-correspondence} from the unital Kre\u\i n \hbox{C*-al}\-ge\-bra $\Bs$ to unital the Kre\u\i n C*-algebra $\As$ is unital left 
Kre\u\i n C*-module ${}_\As\Ms$ over the Kre\u\i n C*-algebra $\As$ equipped with a morphism of unital Kre\u\i n \hbox{C*-algebras} from $\Bs$ to the unital Kre\u\i n C*-algebra $\Bs({}_\As\Ms)$ of adjointable operators on the Kre\u\i n module ${}_\As\Ms$ such that 
$x\cdot \beta(b)=J_\As(J_\As(x)\cdot b)$.  

A \emph{right Kre\u\i n C*-correspondence} from $\Bs$ to $\As$ is similarly defined as a unital right Kre\u\i n \hbox{C*-module} $\Ns_\Bs$ over $\Bs$ equipped with a morphism of unital Kre\u\i n C*-algebras from $\As$ to $\Bs(\Ns_\Bs)$ such that $\alpha(a)\cdot x=J_\Bs(a\cdot(J_\Bs x))$. 

A \emph{morphism of Kre\u\i n C*-correspondences} is a map $\Phi:{}_\As\Ms_\Bs\to{}_\As\Ns_\Bs$ between right (respectively left) Kre\u\i n 
C*-correspondences such that 
\begin{equation*}
\Phi(a\cdot x\cdot b)=a\cdot \Phi(x)\cdot b, \quad \forall a\in \As, \ b\in \Bs, \ x\in \Ms.
\end{equation*} 
\end{definition}

\begin{remark}
The previous definition entails that a left Kre\u\i n C*-correspondence is actually a unital bimodule ${}_\As\Ms_\Bs$ over the the unital Kre\u\i n C*-algebras $\As$ and $\Bs$ (with $\As$-valued inner product) such that there exists at least one fundamental symmetry $J$ of ${}_\As\Ms$ and  fundamental symmetries $\alpha$ of $\As$, $\beta$ of $\Bs$ that satisfy the compatibility condition $J(axb)=\alpha(a)J(x)\beta(b)$, for all $x\in \Ms$, $a\in \As$ and $b\in \Bs$. 

Clearly we have categories of morphisms of right (respectively) Kre\u\i n \hbox{C*-cor}\-re\-spon\-dences under the usual composition of morphisms. 
\xqed{\lrcorner} 
\end{remark}

The following definition and theorems incorporate (and generalize to the case of Kre\u\i n C*-modules over Kre\u\i n \hbox{C*-algebras}) the notion of tensor product of Kre\u\i n spaces and Kre\u\i n C*-modules over \hbox{C*-algebras} developed in R.Tanadkithirun's senior undergraduate 
project~\cite{T}. 
\begin{definition}
The \emph{internal tensor product of two right Kre\u\i n C*-cor\-re\-spon\-dences} ${}_\As\Ms_\Bs$ and ${}_\Bs\Ns_\Cs$ is defined as a left $\As$-linear right $\Cs$-linear and $\Bs$-balanced map $\otimes: {}_\As\Ms_\Bs \times {}_\Bs\Ns_\Cs\to {}_\As\Ts_\Cs$ with values into a right Kre\u\i n 
\hbox{C*-cor}\-respondence ${}_\As\Ts_\Bs$ from $\As$ to $\Bs$ such that the following universal factorization property is satisfied: 

for every left $\As$-linear right $\Cs$-linear and $\Bs$-balanced function\hbox{ $\phi:\Ms\times \Ns\to \Qs$} with values into a Kre\u\i n 
\hbox{C*-correspondence} ${}_\As\Qs_\Cs$ from $\As$ to $\Cs$, there exists a unique morphism $\Phi:\Ts\to\Qs$ of Kre\u\i n 
\hbox{C*-cor}\-respondences  such that $\Phi\circ \otimes=\phi$. 
\end{definition}

\begin{theorem}
Tensor products af right Kre\u\i n C*-correspondences exist and are unique up to isomorphism in the category of morphisms of Kre\u\i n C*-cor\-re\-spon\-dences. 
\end{theorem}
\begin{proof}
The unicity up to isomorphism is a standard consequence of a definition via universal factorization properties. For the proof of existence, consider the  fundamental decompositions $\Ms=\Ms_+\oplus \Ms_-$ and $\Ns=\Ns_+\oplus\Ns_-$ induced by a pair of fundamental symmetries $J_\Ms$ and $J_\Ns$ that are compatible with three fundamental symmetries $\alpha$ of $\As$, $\beta$ of $\Bs$ and $\gamma$ of $\Cs$. 

Using the algebraic tensor product of bimodules and the canonical isomorphism of bimodules 
\begin{multline}\label{eq: dec}
(\Ms_+\oplus \Ms_-)\otimes_\Bs(\Ns_+\oplus\Ns_-)\simeq 
\\
[(\Ms_+\otimes_\Bs\Ns_+)\oplus (\Ms_-\otimes_\Bs\Ns_-)]\oplus  
[(\Ms_+\otimes_\Bs\Ns_-)\oplus (\Ms_-\otimes_\Bs\Ns_+)], 
\end{multline}
we have that $J_\Ms\otimes_\Bs J_\Ns$ is a fundamental symmetry of $\Ms\otimes_\Bs\Ns$ that induces the previous decomposition and is compatible with the left action of $\As$ and the right action of $\Cs$. 

By universal factorization property (two times), we can define on $\Ms\otimes_\Bs\Ns$ a unique $\Cs$-valued inner product such that, 
for all $x_1,x_2\in \Ms$ and $y_1,y_2\in \Ns$,  
\begin{equation*}
\mip{x_1\otimes_\Bs y_1}{x_2\otimes_\Bs y_2}^{\Ms\otimes_\Bs\Ns}_\Cs:=\mip{y_1}{\mip{x_1}{x_2}^\Ms_\Bs\cdot  y_2}^{\Ns}_\Cs. 
\end{equation*}
The following property holds for this inner product 
\begin{equation*}
\gamma(\mip{x_1\otimes_\Bs y_1}{x_2\otimes_\Bs y_2}_\Cs)=
\mip{(J_\Ms\otimes_\Bs J_\Ns)(x_1\otimes_\Bs y_1)}{(J_\Ms\otimes_\Bs J_\Ns)(x_2\otimes_\Bs y_2)}_\Cs
\end{equation*}
and the algebra $\As$ acts by adjointable operators on the left. 

The inner product so defined on $\Ms\otimes_\Bs\Ns$ induces on each one of the direct summands of the even part of the decomposition in 
formula~\eqref{eq: dec} the structure of Hilbert C*-module and, for the summands of the odd part, the structure of anti-Hilbert C*-module over the same C*-algebra $(\Cs, \dag_\gamma)$.  
\end{proof}

\begin{theorem}
There is a weak category $\Mf_\bullet$ whose objects are unital Kre\u\i n \hbox{C*-algebras} $\As,\Bs,\dots$; whose arrows are right Kre\u\i n 
C*-cor\-re\-spon\-dences; and whose composition is obtained by internal tensor product of Kre\u\i n C*-cor\-re\-spon\-dences. 
In a totally similar way, we have a weak category ${}_\bullet\Mf$ of left Kre\u\i n C*-correspondences under internal tensor product. 
\end{theorem}
\begin{proof}
The associativity of composition modulo isomorphism is assured using the universal factorization property. The (weak) identities are given the Kre\u\i n \hbox{C*-algebras} $\As$ considered as right Kre\u\i n C*-correspondences over themselves when equipped with their standard right inner product 
$\mip{a_1}{a_2}_\As:=a_1^*a_2, \quad  a_1,a_2\in \As$.
\end{proof}

\begin{remark} \label{rem: morita}
The previous categories $\Mf_\bullet$ and ${}_\bullet\Mf$ are actually 2-categories considering as 2-arrows the morphisms of Kre\u\i n C*-correspondences with their usual functional composition as composition over 1-arrows and their internal tensor product as composition over objects. 

This pair of weak 2-categories is the ``Kre\u\i n counterpart'' to the usual \hbox{2-cat}e\-gories of right and left C*-correspondences and their (common) subcategory of \hbox{1-isomorphisms}, that consists of Kre\u\i n \hbox{C*-bimodules} ${}_\As\Ms_\Bs$ that are full and satisfy the imprimitivity condition 
${}_\As\mip{x}{y}z=x\mip{y}{z}_\Bs$, for all $x,y,z\in \Ms$, is the ``Kre\u\i n counterpart'' of the Morita-Rieffel weak category of imprimitivity 
\hbox{C*-bimodules} that describe the (strong) Morita equivalence between C*-algebras\footnote{For additional details on the Morita-Rieffel categories and strong Morita equivalence see for example the review sections in~\cite{BCL} and the references therein.} and hence we have a theoretical background capable of discussing the notion of ``\emph{Kre\u\i n-Morita equivalence}'' at least in the context of Kawamura's Kre\u\i n 
\hbox{C*-algebras}. 
\xqed{\lrcorner}  
\end{remark}

\begin{remark} 
The previous categories (exactly as their C*-counterparts) are not equipped with involutions: the contragredient of a right correspondence is a left correspondence, but usually not another right correspondence.\footnote{Recall that, given a bimodule ${}_\As\Ks_\Bs$, over involutive algebras $\As,\Bs$, its contragredient bimodule ${}_\Bs\cj{\Ks}_\As$ is the same Abelian group $\cj{\Ks}:=\Ks$ with left/right actions defined via $b\cdot \cj{x}\cdot a:=\cj{a^*xb^*}$, for all $a\in \As$, $b\in \Bs$, $x\in \Ks$. For a right Kre\u\i n C*-correspondence $\Ks_\Bs$ with right inner product $\mip{x}{y}_\Bs$, for $x,y\in \Ks$, the contragredient ${}_\Bs\cj{\Ks}$ is naturally a left Kre\u\i n C*-correspondence via 
${}_\Bs\mip{\cj{x}}{\cj{y}}:=\mip{x}{y}_\Bs$, for all $\cj{x},\cj{y}\in \cj{\Ks}$, A similar statement holds for left correspondences.}
More interesting notions of ``bivariant'' Kre\u\i n C*-bimodules will be developed elsewhere. 
\xqed{\lrcorner} 
\end{remark}

\begin{example} \label{ex: spin}
Whenever the time-orientable space-orientable (compact) semi-Riemannian even-dimensional manifold $M$ admits a spinorial structure, or more generally a spin$^c$ structure, (see details in H.Baum~\cite{Ba}) the family $\Gamma(S(M))$ of continuous section of a given complex spinor bundle 
$S(M)$ becomes a Kre\u\i n-Morita equivalence Kre\u\i n C*-bimodule between $C(M)$ (on the right) and the Kre\u\i n \hbox{C*-algebra} 
$\Gamma(\CCl(M))$ on the left.\footnote{A similar statement holds in the odd-dimensional case if the Clifford algebra $\Gamma(\CCl(M))$ is replaced by its even part $\Gamma(\CCl^+(M))$, see~\cite[section~9.2]{FGV}.}
Its contragredient Kre\u\i n \hbox{C*-bimodule} $\Gamma(S(M))^*$ is isomorphic to the Kre\u\i n C*-bimodule of sections of the dual spinor bundle $S(M)^*$ and we have $\Gamma(S(M))\otimes_{C(M)}\Gamma(S(M))^*\simeq\Gamma(\Lambda^\CC(M))$ as tensor product of Kre\u\i n C*-bimodules. 
\xqed{\lrcorner} 
\end{example}

\section{Outlook}

The discussions of duality and of spectral theory, via suitable ``Kre\u\i n bundles'', for some ``commutative'' subclasses of the Kre\u\i n C*-modules here defined, will be dealt with in future works.\footnote{See anyway~\cite{BBL} for some elementary results in the case of commutative Kre\u\i n 
C*-algebras.} 

The notion of Kre\u\i n C*-module over a Kre\u\i n C*-algebra that we presented here, although interesting as a first step to explore some of the issues in semi-definite situations, is still too elementary to be fully useful for general applications to non-commutative spectral geometry, at least whenever the semi-Riemannian geometry involved presents topological obstructions to orientability, either in spacelike or in timelike sense (or both). Since global fundamental symmetries in Kre\u\i n \hbox{C*-algebras} 
are remnants of the fundamental decompositions of the Kre\u\i n spaces on which they are faithfully represented, their existence in situations coming from semi-Riemannian geometry seems to be a consequence of such global topological conditions of orientability and it is likely that a more general definition of a complete semi-definite analog of C*-algebras might be necessary to deal with such cases.  
A possible line of attack would be to define semi-definite modules that are direct summand submodules of our ``free-splitting'' Kre\u\i n modules over a C*-algebra (eliminating the topological obstruction on orientability via ``embedding'' into a wider environment exactly as we usually do in the case of projective but non-free modules) and redefine Kre\u\i n C*-algebras as compressions of the ``free-splitting'' Kawamura case. 
We might explore these and other possibilities in subsequent work. 

A more immediately achievable important goal (especially in view of applications to examples of semi-Riemannian geometries related to relativistic physics) is the removal of the unitality (compactness) requirements in the definitions of Kre\u\i n C*-modules and Kre\u\i n C*-algebras.   

Our main long-term interest is to formulate notions of semi-definite involutive operator algebraic environments that are suitable, as a (topological) background, for the development of non-commutative geometry and spectral triples in a completely general semi-definite situation. 


\medskip

\emph{Notes and acknowledgments} 
The paper originates from and further elaborates on material presented in S.~Kaewunpai master thesis~\cite{Ka} as well as on R.~Tanadkithirun and A.Atchariyabodee undergraduate senior projects~\cite{T,A}. Thanks to Starbucks Coffee at the 1$^\text{st}$ floor of the Emporium Suites Tower for the welcoming environment where this research work has been discussed and written. 
The authors thank the two anonymous referees of the paper.

\end{document}